\documentclass[sn-mathphys-num]{sn-jnl}

\usepackage{graphicx}%
\usepackage{multirow}%
\usepackage{amsmath,amssymb,amsfonts, tikz-cd}%
\usepackage{amsthm}%
\usepackage{mathrsfs}%
\usepackage[title]{appendix}%
\usepackage{xcolor}%
\usepackage{textcomp}%
\usepackage{manyfoot}%
\usepackage{booktabs}%
\usepackage{algorithm}%
\usepackage{algorithmicx}%
\usepackage{algpseudocode}%
\usepackage{listings}%

\newtheorem{theorem}{Theorem}[section]

\newtheorem{definition}[theorem]{Definition}

\newtheorem{lemma}[theorem]{Lemma}

\newtheorem{remark}[theorem]{Remark}

\begin{document}

\title{Norm Principles for  ${\rm GO}(q)$ and ${\rm Spin}(q)$}


\author*[1]{\fnm{Priyabrata} \sur{Mandal}}\email{p.mandal@manipal.edu}



\affil*[1]{\orgdiv{Department of Mathematics}, \orgname{Manipal Institute of Technology, Manipal
Academy of Higher Education}, \orgaddress{ \city{Manipal}, \postcode{576104}, \state{Karnataka}, \country{India}}}




\abstract{
	 The aim of this article is to discuss the proof that the norm principle holds for the group of similitudes ${\rm GO}(q)$ and the spinor group ${\rm Spin}(q)$ for a quadratic form $q$ defined over $F$.
}




\maketitle

\begin{section}
{Introduction:}
In this paper, we explore norm principles for connected reductive algebraic groups defined over a field $F$, where char $F \neq 2$. Consider a commutative linear algebraic group $T$ defined over $F$, and let $K/F$ be a finite separable field extension. The $norm\ homomorphism\ N_{K/F}$, is defined as
\[
N_{K/F}: T(K) \to T(F)\]
by sending $t \mapsto \prod_{\gamma } \gamma(t)$, where $\gamma \in \frac{{\rm Gal}(F^{\rm sep}/F)}{{\rm Gal}(F^{\rm sep}/K)}$ and $F^{\rm sep}$ be the separable closure of $F$.

\

If $T$ = $\mathbb {G}_m$, then $N_{K/F}$ is the usual field norm $N_{K/F} : K^\times \to F^\times$.

Suppose $G(F)$ is a connected linear algebraic group over $F$, and let $T(F)$ be a commutative linear algebraic group also defined over $F$. Let $\phi : G(F) \to T(F) $ be an algebraic group homomorphism defined over $F$. For a finite separable extension $K/F$, let $\phi_K$ be the algebraic group homomorphism from $G(K)$ to $T(K)$. 

Consider the following diagram :\begin{center}
\begin{tikzcd}
	G(K) \arrow[r, "\phi_K"] 
	& T(K) \arrow[d, "N_{K/F}"] \\
	G(F) \arrow[r,"\phi"]
	& T(F)
\end{tikzcd}
\end{center}
The norm principle is said to hold for $\phi$ if  
\begin{equation}\label{eq1} 
N_{K/F}(\phi_K (G_K)) \subseteq \phi(G(F)).
\end{equation} 
If, for every separable field extension $K/F$, equation (\ref{eq1}) holds, then the norm principle is said to hold for $\phi : G \to T$.

The objective of this paper is to discuss the proof that the norm principle applies to the group of similitudes ${\rm GO}(q)$ and the spinor group ${\rm Spin}(q)$ for a quadratic form $q$ defined over $F$.

\end{section}
\section{Preliminaries:} \label{prelim}
All fields under consideration are assumed to have a characteristic distinct from $2$. We denote the set of all non-zero elements of $F$ by $F^\times$, i.e., $F^\times=F \setminus \{0\}$.
\medskip


Let $V$ be any finite dimensional $F$-vector space.
Recall that, by a {\it bilinear form} $b$ on $V$, we mean a map $b: V \times V \to F$ satisfying the following: 
\begin{enumerate}

    \item $b(\alpha x+ \beta y,z)=\alpha b(x,z)+ \beta b(y,z).$
    \item $ b(x,\beta y+ \gamma z)=\beta b(x,y)+\gamma b(x,z)$, for all $x,y,z \in V$ and for $\alpha, \beta, \gamma \in F$.
\end{enumerate}
A bilinear form $b$ is said to be symmetric if $b(x,y)=b(y,x)$. In this paper, by a bilinear form we mean a symmetric bilinear form.

A map $q:V \to F$ is said to be a {\it quadratic form} on $V$ if 
\begin{enumerate}
    \item $q(\alpha x)=\alpha^2 q(x)$.
    \item $b_q(x,y)=\frac{1}{2}(q(x+y)-q(x)-q(y)$ is a (symmetric) bilinear form on $V$.
\end{enumerate}
We call $b_q$ be the bilinear form associated to the quadratic form $q$. Also, given a bilinear form $b$ on $V$, one can associate a quadratic form $q_b$ to $b$ by defining $q_b(x):= b(x,x)$.
We refer to (Chapter 1, \cite{sch}) for more details.
The form $q:V\to F$ is said to be {\it regular} if $b_q:V \times V \to F$ is nondegenerate.
We call $(V,q)$ a quadratic space. Two quadratic spaces $(V_1, q_1) \text{ and } (V_2,q_2)$ is said to be {\it isometric} if there is an isomorphism $\sigma: V_1 \to V_2$ satisfying $q(x)=q_1(\sigma(x))$. We then denote $V_1 \cong V_2$. For a $2$ dimensional quadratic space $(V,q)$, if $V \cong \langle 1,-1 \rangle $, then we call $V$ a {\it hyperbolic plan} and denote by $\mathbb{H}$. A {\it hyperbolic space} is an orthogonal sum of hyperbolic planes. A quadratic form $q$ can always be diagonalized as $q=\langle d_1\rangle \perp \dots\perp \langle d_n \rangle = \langle d_1,\dots,d_n \rangle$.

\medskip

For a quadratic space $(V,q)$ over $F$, if there exists a non-zero $x \in V$ such that $q(x)=0$, then we say $q$ is isotropic over $F$. Otherwise, $q$ is called {\it anisotropic}. For $n \in \mathbb N$, let $n.q$ denotes the $n$-fold orthogonal sum of $q$. If $n.q$ is isotropic, then we say $q$ is {\it weakly isotropic} over $F$. A $2$-dimensional quadratic form $q$ is said to be {\it hyperbolic} if $q \cong \langle 1,-1 \rangle$. By Witt decomposition theorem (see Chapter 1, Theorem 4.1, \cite{lam}), any quadratic form $q$ can be written as $q=q_h \perp q_a$, where $q_h$ denotes the hyperbolic part of $q$ and $q_a$ denotes the anisotropic part of $q$.
\medskip

 




For a quadratic space $(V,q)$ over $F$, let $D(q)$ denote the set of elements in $  F^\times$ represented by $q$, i.e.,
$D(q)=\{ x \in F^\times |\ \text{there exists } v \in V \text{such that }q(v)=x \}.$ 

.  

    \begin{lemma}\label{square class}
        Let $(V,q)$ be a quadratic space over $F$. If $\alpha,\ x \in F^\times$, then $x \in D(q)$ if and only if $\alpha^2 x \in D(q)$.
    \end{lemma}
    \begin{proof}
        Suppose, $x \in D(q)$, there exists $v \in V$ such that $q(v)=x$. For $\alpha \in F^\times$, $\alpha v \in V$. Since $q$ is a quadratic form, we have $q(\alpha.v)= \alpha^2 q(v)= \alpha^2x$. Therefore, $\alpha^2x \in D(q)$.
        Conversely, if $\alpha^2 x \in D(q)$, then there exists $v \in V$ such that $q(v)=\alpha^2x$. Hence, $q(\frac{1}{\alpha})v=\frac{1}{\alpha^2}q(v)=x$. Thus $x\in D(q)$.
    \end{proof}
   \begin{remark} \normalfont
       
   From the above lemma \ref{square class}, it is easy to see that $D(q)$ consists of a union of cosets of $F^\times/F^{\times 2}$. In general, $D(q)$ need not be a subgroup of $F^\times$. If it forms a subgroup, then we call $q$ a {\it group form} over $F$ (see Chapter 1, \S 2, \cite{lam}). However, every element in $D(q)$ has an inverse in $D(q)$. Indeed, for $x \in D(q)$, one can write 
    \[x^{-1}= (x^{-1})^2.\ x \in D(q).\]
     \end{remark}
    






\medskip 

We next discuss construction of Witt rings.
Consider a commutative cancellation monoid $(M,+)$. Define a relation $ \sim $ on $M \times M$ by
\[(x,y)\ \sim\ (x',y') \iff x+y'=x'+y \in M.\] 
Now consider $(M \times M)/ \sim$. We denote the equivalence class of $(x,y)$ by $[x,y]$. Define an addition by \[[x,y]+ [x',y']:= [x+x',y+y']\] 
This is clearly well defined, associative and commutative.
The classes $[x,y], [y,x]$ are additive inverses to each other, making it a group, called a {\it Grothendieck group} and denoted by ${\rm Groth} (M)$. If $M$ has a multiplication which makes it a semiring, then by defining the product
\[[x,y][x',y']:=[xx'+yy', xy'+yx'],\] ${\rm Groth}(M)$ turns into a ring.
\medskip


\begin{definition}
Let $M(F)$ be the set of all isometry classes of regular quadratic forms over $F$. Then $\widehat{W}(F)= Groth(M(F))$ is called the $Witt$-$Grothendieck$ ring of quadratic forms over the field $F$.
The Witt ring $W(F)$ is obtained by quotienting the ring $\widehat{W}(F)$ by the ideal generated by the hyperbolic spaces, namely $\mathbb{Z \cdot H}$, $i.e$,
$W(F)= \widehat{W}(F) / \mathbb{Z \cdot H}$.
\end{definition}
\medskip 
\begin{theorem}
    The elements of $W(F)$ are in one-one correspondence with the isometry classes of all anisotropic forms over $F$.
\end{theorem}
\begin{proof}
    Any element in $\widehat W(F)$ can be written as $[q_1]-[q_2]$, where $[q_1], [q_2]$ are isometry classes of regular quadratic forms of $q_1$ and $q_2$ respectively. We claim that any element in $W(F)$ is of the form $[q]$, where $q$ is a regular quadratic form over $F$. Indeed, $[q_1]-[q_2]= [q_1 \perp (-q_2)]-[q_2 \perp (-q_2)]$. Now, for any scalar $a \in F^\times$, 
    \begin{equation} \label{witt ring equation}
          \langle a \rangle \perp \langle -a \rangle \cong \langle a, -a \rangle \cong a \langle 1,-1 \rangle = a \mathbb H = 0 \in W(F)
    \end{equation}
    which implies $- \langle a \rangle= \langle -a \rangle \in W(F)$. Note that, $[q_2 \perp (-q_2)]$ is a hyperbolic form over $F$ and hence it becomes zero over $W(F)$. Therefore, $[q_1]-[q_2]= [q_1 \perp (-q_2)] \in W(F)$. If there is any common scalar in the diagonalization of $q_1$ and $q_2$, then by the above equation \ref{witt ring equation}, it denote the zero element in $W(F)$. In particular, every element in $W(F)$ is represented by a form $[q]$. By Witt decomposition theorem, every quadratic form $q$ can be written as $q=q_h \perp q_a$, where $q_h$ denotes the hyperbolic part of $q$ and $q_a$ denotes the anisotropic part of $q$. Therefore, $[q]$ and $[q_a]$ denotes the same element in $W(F)$. To prove the one-one correspondence, we need to show if $[q]$ and $[\phi]$ be two regular quadratic forms such that $[q]= [\phi ] \in W(F)$, then $q \cong \phi$. If $[q]= [\phi ] \in W(F)$, then $[q]= [\phi] \perp m \mathbb H \in \widehat W(F)$ for some integer $m \geq 0$. Hence, $q \cong \phi \perp m \mathbb H$ as quadratic forms. Since, $q$ is anisotropic, it cannot contain a hyperbolic subpart in it, and so $m=0$. Thus, $q \cong \phi$. This completes the proof.
\end{proof}


\medskip

\section{Norm principles} 
Consider a field extension $K$ of $F$. Let $V$ be a vector space over $F$. Then $V$ can be considered as a vector space over $K$, given by $V_K:=V \otimes_F K$. Thus any $F$-quadratic space $(V,q)$ can also considered as a $K$-quadratic space $(V_K, q_K)$, where $q_K=q \otimes_F K$.
\subsection{Scharlau's norm principle}
Suppose $K/F$ is a field extension. If there exists a non-zero linear functional $s:K \to F$, then for any $K$-quadratic space $(W, \phi)$, we can construct a $F$-quadratic space $s_*(W):=(W, s\phi)$. Moreover, if $K/F$ is finite, ${\rm dim}_F\ s_*(W)=[K:F] \cdot {\rm dim}_K(W)$. 
We next state the Frobenius Reciprocity theorem (see \cite{lam}, Chapter 7, Theorem 1.3).
\medskip 

\begin{theorem}\label{Frobenius reciprocity}
Let $K/F$ be a field extension. Let $(V,q)$ be a quadratic space over $F$ and $(W, \phi)$ be a quadratic space over $K$. Then there exists an $F$-isometry
 \[s_*(V_K \otimes _K W) \cong V \otimes _F s_*(W)\]
\end{theorem}
\medskip 

Consider the field extension $K/F$ of degree $n$, where $K=F(x)$. Let $ \{1,x,\dots,x^{n-1} \}$ be an $F$-basis on $K$. So, we have a unique $F$-linear functional $s : K \to F$ given by    $s(1)=1\text{ and } s(x)=s(x^2)=\dots=s(x^{n-1})=0$. Then by (\cite{lam}, Chapter 7, Corollary 2.4), 
\begin{equation} \label{scharlau transfer}
    s_*(\langle 1,-x \rangle  ) =  \langle 1, -N_{K/F}(x) \rangle \in W(F).
\end{equation}



Recall that, two $F$-quadratic forms $q_1$ and $q_2$ are called {\it proportional} if $q_1 \cong \alpha. q_2$ for some $\alpha \in F^\times$ (see \cite{sch}, Chapter 2, Definition 8.4). In particular, if there exists $\alpha \in F^\times$ such that $q_1 \cong \alpha. q_1$, then $\alpha$ is said to be a {\it proportionality (similarity) factor} of $q_1$. For a quadratic form $q$ over a field $F$, consider the set
\[G(q):= \{ \alpha \in F^\times: \alpha.q \cong q\}\]
It is easy to see that $G(q)$ is a subgroup of $F^\times$. It is called the group of {\it proportionality (similarity) factors} of $q$. 
For any $\alpha \in F^{\times 2}$, by the property of quadratic forms, we have $\alpha.q \cong q$. Thus, $\alpha \in G(q)$. In particular, $F^{\times 2} \subseteq G(q)$. We now discuss Scharlau's norm principle.
\medskip

\begin{theorem}[Scharlau] \label{scharlau norm principle}
Let $K$ be a finite field extension of $F$ and $q$ be a regular quadratic form over $F$. Then for any $x \in K^\times$, the following inclusion holds 
    \[x \in G(q_K) \implies N_{K/F}(x) \in G(q).\]
    In other words, 
    \[N_{K/F}(G(q_K)) \subseteq G(q)\]
\end{theorem}
\begin{proof} 
Let $x \in K^\times $. Consider the intermediate field $F(x)$ with $F \subseteq F(x)\subseteq K $. For convenience, let us denote the field $F(x)$ by $E$.

\textbf{Case (i):} Suppose $K/E$ is an even degree extension, say, $[K:E]=2m$. Then $N_{K/E}(x)=x^{2m}$.
By the multiplicative property of norm, we have
    \[N_{K/F}(x)=N_{E/F}\cdot (N_{K/E}(x))=N_{E/F} \cdot (x^{2m})=(N_{E/F}(x^m))^2 \in F^{\times 2} \]
Therefore, $N_{K/F}(x) \in G(q)$.
\medskip 

\textbf{Case (ii):} Suppose, $[K:E]=2m+1$, an odd degree extension. As $x \in G(q_K)$, we have $x.\ q_K \cong q_k$, i.e., $\langle 1, -x \rangle \otimes_K q_K=0 \in W(K)$. By (\cite{lam}, Chapter 7, Theorem 2.5), the map $W(E) \to W(K)$ is injective. 
Hence $\langle 1,-x\rangle\otimes_E q_{E}=0 \in W(E))$ and so $x \in G(q_{E})$. Now, 
\begin{align*}
N_{K/F}(x)= &N_{E/F}\cdot (N_{K/E}(x))\\= &N_{E/F} \cdot (x^{2m+1})\\ 
= &N_{E/F}(x) \cdot  (N_{E/F}(x^m))^2
\end{align*}
Thus if we show that $N_{E/F}(x) \in G(q)$, we are done. So we can assume $K=F(x)$ and let $s: K \to F$ be the unique $F$-linear functional defined as earlier this section. 
Applying the transfer $s_*$ to the equation $\langle 1,-x \rangle   \otimes_K \ q_K =0 \in W(K)$, we get
\begin{align*}
    0=\ & s_*(q_K \otimes _K \langle 1,-x \rangle  _{K})\\ \cong\ & q \otimes _F s_*(\langle 1,-x \rangle  _{K})\ ( \text{by Theorem } \ref{Frobenius reciprocity})\\
    \cong\ & q\ \otimes _F \langle 1,-N_{K/F}(x) \rangle  \ \in W(F)\ (\text{by Equation } \ref{scharlau transfer})
\end{align*}
Hence $\langle 1,-N_{K/F}(x) \rangle \otimes_F q =0 \in W(F)$\ and so $N_{K/F}(x) \in G(q)$.
\end{proof}
\medskip 
\begin{remark}
    Let ${\rm GO}(q)$ be the group of similitudes and $\phi: {\rm GO}(q) \to \mathbb G_m$ be the multiplier homomorphism (see \cite{kmrt}, Chapter 3 for more details). The norm principle for $\phi$ readily follows from the above Theorem \ref{scharlau norm principle} (see \cite{bm}, example 3.3).
\end{remark}
\medskip 

\subsection{Knebusch's norm principle}
Recall the notion of $D(q)$ for a $F$-quadratic space $(V,q)$ defined in section \ref{prelim} as follows
\[D(q)=\{ d \in F^\times |\ \text{there exists } v \in V \text{such that }q(v)=d \}.\]  
In the earlier section \ref{scharlau norm principle}, we discuss the group of similarity factors $G(q)$ with respect to the norm map for a finite field extension. It will be of interest to establish a parallel result for $D(q)$ also. But, the main problem is $D(q)$ need not be a subgroup of $F^\times$. In fact, $D(q)$ need not be closed under multiplication. For example, consider the quadratic form $q= \langle 1,1,1 \rangle = x^2+y^2+z^2$ over $\mathbb Q$. Then clearly, $1,2,2^{-1},14 \in D(q)$. But the product $2^{-1}.14=7 \notin D(q)$ as $7$ can not be expressed as a sum of three squares over $\mathbb Q$. 
\medskip

\begin{theorem} [Knebusch] \label{knebusch norm principle}
Suppose $K/F$ is a finite extension of degree $n$ and $q$ is a regular quadratic form over $F$. Let $x \in K^\times.$ If $x \in D(q_K)$, then $N_{K/F}(x) \in D^n(q)$\ (i.e., it is a product of $n$ elements of $D(q))$.
\end{theorem}
\begin{proof}
If $q$ is isotropic over $F$, then by (\cite{lam}, Chapter 1, Theorem 3.4), $D(q)= F^\times$. Hence, $N_{K/F}(x) \in F^\times = D(q)$.
Suppose $q$ is an anisotropic form over $F$. For $m \geq 1$, define $D^m(q)$ be the set of products of $m$ elements of $D(q)$. Let $d \in D(q)$. Then for any $c \in F^\times$,
 \[c^2= d \cdot ((c/d)^2 \cdot d) \in D(q) \cdot D(q) = D^2(q).\]
 Thus, $F^{\times 2} \subseteq D^2(q)$. More generally, $ F^{\times 2m} \subseteq D^{2m}(q)$. 
 
 We now prove the theorem by induction on $[K:F] = n$. If $n=1$, then $K=F$ and we are done. Assume $n \geq 2 $ in the following.
 
 \textbf{Case i:} Let $x \in  F$. If $[K:F]= 2m$, then 
 \[N_{K/F}(x)=x^{2m} \in F ^{\times 2m} \subseteq D^{2m}(q).\]
 If $[K:F]=2m+1,$ then $N_{K/F}(x)= x^{2m+1} \in x \cdot F^{\times 2m}$. Since $K/F$ is an odd degree extension, by an application of Springer's theorem (see \cite{lam}, Chapter 7, Corollary 2.9), $x \in D(q_K) \implies x \in D(q)$. Hence   \[N_{K/F}(x) \in x \cdot  F^{\times 2m} \subseteq D(q) \cdot D^{2m}(q)=D^n(q)\]
\textbf{Case ii:} Suppose $x \not\in F.$ Let $E=F(x)$ and consider $ F \subseteq E \subseteq K$ with $[K:E]=m,\ [E:F]=m',\ n=mm'.$

Suppose $E \subsetneq K$. Then $m>1$ and by the above case i,
 $N_{K/E}(x) \in D^m(q_E)$ 
 Therefore, $N_{E/F} \cdot  N_{K/E}(x) \in N_{E/F}( D^m(q_E))$
 Hence, \begin{equation} \label{p1}
 N_{K/F}(x) \in N_{E/F}( D^m(q_E)) 
 \end{equation}
As $m> 1$ and so $m'<n$, by induction hypothesis on $E/F$, we have\[ N_{E/F}(x) \in D^{m'}(q),\ \text{for each}\ x \in D(q_E).\] 
Hence, using \ref{p1}, $N_{K/F}(x) \in D^{mm'}(q)=D^n(q)$.

Suppose $E=K$. Let $p(t)$ be the minimal polynomial of $x$ over $F$, so\[
K=F(x) \cong F[t]/(p(t))\]
As $D(q_K)$ is closed under inverses, and $x \in D(q_K)$, we have $x^{-1} \in D(q_K)$.
So there exists $f_1,\dots,f_d \in F[t]$, satisfying \[ q(f_1(x),\dots ,f_d(x)) = x^{-1}.\] 
Since $p(t)$ is the minimal polynomial of $x$ over $F$, we can write
\begin{equation}\label{pm2}
t \cdot q(f_1(t),\dots,f_d(t))= 1+p(t)h(t)
\end{equation}
  where $d= {\rm dim}\ q,\ h(t), f_i(t) \in F[t]$ with $r:= {\rm max} \{{\rm deg}(f_i) \} \leq n-1$. Since, $q$ is anisotropic, \ref{pm2} shows that $n_0:= {\rm deg}(h)=2r+1-n \leq 2(n-1)+1-n = n-1.$

If $h(t)=c \cdot h_1(t) \cdots h_s(t)$, where $c \in F^\times$ and the $h_i$'s are monic irreducible polynomials in $F[t]$, then $c$ is the leading co-efficient of $ 1+p(t)h(t)$. Using \ref{pm2}, we have $c \in D(q)$.
Now if $s=0$ in the decomposition of $h$, i.e, if $h(t)$ is a  constant polynomial $c$, then $n_0=0$ and \[
N_{K/F}(x)= (-1)^n p(0)= (-1)^{n+1}h(0)^{-1} = (-1)^{n+1}c^{-1} = (-1)^{2r+1+1}c^{-1}= c^{-1} \]
Since $c \in D(q)$, we have $N_{K/F}(x) \in D(q)$.
Suppose $s \geq 1$ in the decomposition of $h$, and  $x_i$ is a zero of $h_i$ in an algebraic closure of $F$. Then from (\ref{pm2}), we have \[
x_i^{-1}= q(f_1(x_i),\dots,f_d(x_i)) \in D(q_{F(x_i)})\].
Since $[F(x_i):F] \leq {\rm deg}\ h \leq n-1$, by induction hypothesis on $F(x_i)/F$,\[
N_{F(x_i)/F}(x_i)= (-1)^{deg\ h_i}h_i(0) \in D^{deg\ h_i}(q).\]
Taking the product of these over $i$, we get\[
 (-1)^{n_0}h_1(0) \dotsi h_s(0)= (-1)^{n_0} c^{-1} h(0) \in D^{n_0}(q).\]
Since $c^{-1} \in D(q)$, we have\[
(-1)^{n_0} h(0) \in D^ {(n_0+1)} (q).\]
As $n_0+1 \equiv n\ (mod\ 2)$ and recalling that $ F^{\times 2m} \subseteq D^{\times 2m}(q)$, we have\[
(-1)^{n+1} h(0) \in D^ {n} (q).\] Therefore, $N_{K/F}(x)= (-1)^n p(0)= (-1)^{n+1}h(0)^{-1} \in D^n(q)$. This completes the proof.
\end{proof}
\begin{remark}
    Let $(V,q)$ be a quadratic space over $F$ and let $\Gamma^+ (V,q)$ be the even Clifford group of $(V,q)$. Consider the spinor norm homomorphism ${\rm Sn}: \Gamma^+(V,q) \to \mathbb G_m$. The image of ${\rm Sn}_F$ in $F^\times$ consists of the products of non-zero values of the quadratic form $q$. The kernel of the spinor norm homomorphism is called the spinor group of $(V,q)$ and denoted by ${\rm Spin}(V,q)$. The norm principle for ${\rm Spin}(V,q)$ follows the above Theorem \ref{knebusch norm principle} (see \cite{bm}, Example 3.2).
\end{remark}

\end{document}